\documentclass[12pt,reqno]{amsart}
 
\usepackage{url} 
\usepackage{pdfsync}
\usepackage[all,poly]{xy}
\usepackage{amssymb}

\usepackage{mathrsfs}

\DeclareMathAlphabet{\mathpzc}{OT1}{pzc}{m}{it}

\usepackage{tikz}
\usetikzlibrary{arrows,decorations.pathmorphing,decorations.pathreplacing,positioning,shapes.geometric,shapes.misc,decorations.markings,decorations.fractals,calc,patterns}

\usepackage{graphicx}

\def\BZ{\mathbb{Z}}

\def\adots{\mathinner{\mkern1mu\raise1.0pt\vbox{\kern7.0pt\hbox{.}}\mkern2mu\raise4.0pt\hbox{.}\mkern2mu\raise7.0pt\hbox{.}\mkern1mu}}

\def\dim{\operatorname{dim}}

\def\prod{\operatorname{prod}}

\newcommand{\Set}[1]{\ensuremath{\mathcal{#1}}}  

\newcommand{\size}[1]{\left\lvert #1\right\rvert}
\newcommand{\qi}[2][q]{[#2]_{#1}}
\newcommand{\qbinom}[3][q]{\genfrac{[}{]}{0pt}{}{#2}{#3}_{#1}}
\def\T{\mathcal T}%
\def\P{\mathcal P}%
\def\F{\mathcal F}%
\def\G{\mathcal G}%

\newcommand{\Dfn}[1]{\emph{#1}} 

\numberwithin{equation}{section}

\newtheorem{Lemma}{Lemma}[section]
\newtheorem{Theorem}[Lemma]{Theorem}
\newtheorem{Proposition}[Lemma]{Proposition}
\newtheorem{Corollary}[Lemma]{Corollary}

\theoremstyle{definition}
\newtheorem{Definition}[Lemma]{Definition}

\newtheorem{Remark}[Lemma]{Remark}

\newtheorem{Example}[Lemma]{Example}

\newcommand{\done}[1]{}

\begin{document}

\setlength{\parindent}{0pt}
\setlength{\parskip}{4pt}

\title[Torsion pairs in cluster tubes]{Torsion pairs in 
cluster tubes}

\author{Thorsten Holm}
\address{Institut f\"{u}r Algebra, Zahlentheorie und Diskrete
  Mathematik, Fa\-kul\-t\"{a}t f\"{u}r Ma\-the\-ma\-tik und Physik, Leibniz
  Universit\"{a}t Hannover, Welfengarten 1, 30167 Hannover, Germany}
\email{holm@math.uni-hannover.de}
\urladdr{http://www.iazd.uni-hannover.de/\~{ }tholm}

\author{Peter J\o rgensen}
\address{School of Mathematics and Statistics,
Newcastle University, Newcastle upon Tyne NE1 7RU, United Kingdom}
\email{peter.jorgensen@ncl.ac.uk}
\urladdr{http://www.staff.ncl.ac.uk/peter.jorgensen}

\author{Martin Rubey}
\address{Endresstr.\ 59/14, 1230 Wien, Austria}
\email{martin.rubey@math.uni-hannover.de}
\urladdr{http://www.iazd.uni-hannover.de/~rubey}


\thanks{{\em Acknowledgement. }This work has been carried out in the framework of the
  research priority programme SPP 1388 {\em Darstellungstheorie} of
  the Deutsche Forschungsgemeinschaft (DFG).  We gratefully acknowledge
  financial support through the grants HO 1880/4-1 and
  HO 1880/5-1. }



\begin{abstract} 

We give a complete classification of torsion pairs in the cluster
categories associated to tubes of finite rank. The classification
is in terms of combinatorial objects called Ptolemy diagrams which
already appeared in our earlier work on torsion pairs in cluster 
categories of Dynkin type A. As a consequence of our classification 
we establish closed formulae enumerating the torsion pairs in cluster 
tubes, and obtain that the torsion pairs in cluster tubes exhibit a 
cyclic sieving phenomenon.
\end{abstract}

\maketitle


\begin{center}
{\em Dedicated to Idun Reiten on the occasion of her 70th birthday.}
\end{center}

\vskip0.5cm


\done{I tried to make characterisation (iv) consistent with the new
  notation, but it's not completely trivial, because we have to allow
  pairs $\big((i,j), W\big)$ there, where $(i,j)$ is not a vertex
  because $j-i = 1$.}%

\section{Introduction}
\label{sec:introduction}

Torsion pairs form a fundamental and important aspect of representation
theory. Their definition for abelian categories goes back to a paper 
by S.\,Dickson \cite{Dickson} from the mid 1960s. Torsion pairs play a 
particularly prominent role in the context of tilting theory.
The latter has been introduced in the early 1980s in the
seminal papers by S.\,Brenner and M.\,Butler \cite{Brenner-Butler} and by 
D.\,Happel and C.\, Ringel \cite{Happel-Ringel} and has remained a key topic
in the representation theory of finite-dimensional algebras
ever since.   

Many important developments in modern representation theory take place 
not in the module categories themselves but in related triangulated  
categories like stable and derived categories or cluster categories. 
The latter have been introduced in \cite{BMRRT} and
provide a highly successful categorification of Fomin and 
Zelevinsky's cluster algebras. 

In this paper we consider cluster categories coming from tubes
and as main result give a complete classification of the torsion pairs
in these {\em cluster tubes}. This classification is in terms of 
certain combinatorial objects called Ptolemy diagrams which have been
introduced by P. Ng \cite{Ng} and already
occured in the classifications of torsion pairs in cluster categories
of Dynkin type $A$ \cite{HJR-Ptolemy} and $A_{\infty}$ \cite{Ng}.

It turns out that for a given rank there are only finitely 
many torsion pairs in the cluster tube and as a second main result we provide 
a complete enumeration for torsion pairs in cluster tubes, i.e.\ we 
establish a formula for the generating function and from that deduce a
closed formula for the number of torsion pairs in cluster tubes
of a given rank. 

Tubes arise naturally and frequently in representation theory, 
not least as components of the Auslander-Reiten quivers of many finite
dimensional algebras. Tubes are stable translation quivers of the 
form $\mathbb{Z}A_{\infty}/(n)$ for a natural number $n$, i.e.\ they 
have the 
shape of a cylinder which is infinite on one side. For details on
how tubes arise in the Auslander-Reiten theory of algebras of tame
representation type we refer to C.\,M.\,Ringel's book \cite[Chapter 3]{Ringel}.
Taking the additive hull of the mesh category of a tube as above
one obtains an abelian category $\mathsf{T}_n$, called a {\em tube 
category of rank $n$}. 
It naturally occurs in the representation 
theory of extended Dynkin quivers $\tilde{A}$; namely, the category 
$\mathsf{T}_n$ is equivalent to the category of nilpotent representations 
of a cyclic quiver with $n$ vertices and edges cyclically oriented
(cf. \cite[sec.\ 3.6\,(6)]{Ringel}).

In a recent paper, K. Baur, A. Buan and R. Marsh \cite{BBM}
have classified torsion
pairs in this abelian category $\mathsf{T}_n$.
On the other hand, our results in this paper concern the triangulated 
cluster category (to be defined below) attached to the tube category. 
Although both categories have the same Auslander-Reiten quiver
it turns out that their torsion theories are different. In particular, 
our results in this paper do not overlap with the results in \cite{BBM}.

The abelian category $\mathsf{T}_n$ has many nice properties, in particular
it is hereditary and Hom-finite.
This allows to form, as in \cite{BMRRT}, the cluster category attached
to $\mathsf{T}_n$ as the orbit category
$$\mathsf{C}_{n} := D^b(\mathsf{T}_n)/(\tau^{-1}\circ \Sigma)$$
where $D^b(\mathsf{T}_n)$ is the bounded derived category of the 
hereditary abelian category $\mathsf{T}_n$ and $\tau$ and $\Sigma$
are the Auslander-Reiten translation and the suspension on 
$D^b(\mathsf{T}_n)$, respectively. From a result of B. Keller \cite{Keller}
it follows that this orbit category is a 
triangulated category. 

This triangulated category $\mathsf{C}_{n}$ is called 
the {\em cluster tube of rank $n$}. 

We shall prove the following result in this paper which provides a
complete classification of torsion pairs in cluster tubes. In fact, 
as an important step on the way we show in Proposition 
\ref{prop:torsion-vs-Ptolemy} that for each torsion pair
in $\mathsf{C}_n$ precisely one of the subcategories $\mathsf{X}$
and $\mathsf{X}^{\perp}$ has only finitely many indecomposable objects. 

For the definition of Ptolemy diagrams and levels 
we refer to Section \ref{sec:model} and for wings to Section 
\ref{sec:classification} below.

\begin{Theorem} 
There are bijections between the following sets:
\begin{enumerate}
\item[{(i)}] 
Torsion pairs $(\mathsf{X},\mathsf{X}^{\perp})$
in the cluster tube $\mathsf{C}_n$ such that $\mathsf{X}$ has only
finitely many indecomposable objects. 
\item[{(ii)}] $n$-periodic Ptolemy diagrams $\mathcal{X}$ of the
$\infty$-gon such that all arcs in $\mathcal{X}$ have length at most $n$.
\item[{(iii)}] Collections 
$\{([(i_1,j_1)],[W_1]),\ldots,([(i_r,j_r)],[W_r])\}$
of pairs consisting of vertices $[(i_{\ell},j_{\ell})]$
of level $\le n-1$ in the AR-quiver of $\mathsf{C}_n$ and 
subsets $[W_{\ell}]\subseteq W[(i_{\ell},j_{\ell})]$ of their wings 
such that for any different $k,{\ell}\in\{1,\ldots,r\}$ we have
$$\Sigma\, W[(i_k,j_k)]\cap W[(i_{\ell},j_{\ell})]
=\emptyset
$$
and the $n$-periodic collection $W_{\ell}$ of arcs corresponding to $[W_{\ell}]$ is a Ptolemy diagram (in which every arc is 
overarched by some arc from the collection corresponding 
to $[(i_{\ell},j_{\ell})]$). 
\end{enumerate}
\end{Theorem}

Along the way we also obtain an independent proof of 
the following result of A. Buan, R. Marsh and D. Vatne.

\begin{Corollary}[\cite{BMV}]
The cluster tubes $\mathsf{C}_{n}$ do not contain any 
cluster tilting objects. 
\end{Corollary}

One of the outcomes of the above classification result is that for a given
rank $n$ there are only finitely many torsion pairs in the cluster
tube $\mathsf{C}_{n}$. In the final section we count 
the number of torsion pairs and obtain the following enumeration
result.

\begin{Theorem}
The number of torsion pairs in the cluster tube $\mathsf{C}_n$ is
equal to 
$$
  \T_n = \sum_{\ell\geq 0}%
  2^{\ell+1}\binom{n-1+\ell}{\ell}\binom{2n-1}{n-1-2\ell}.
  $$
\end{Theorem}

In Section \ref{sec:enumerate} we also present refined counts and
determine the number of torsion pairs up to Auslander-Reiten
translation.  Moreover, we show that the set of torsion pairs in the
cluster tube $\mathsf{C}_n$ together with the cyclic group generated
by the Auslander-Reiten translation exhibits a cyclic sieving phenomenon.

The paper is organised as follows. In Section \ref{sec:torsion} 
we review some 
fundamentals on torsion pairs in triangulated categories, building 
on the paper by O.\,Iyama and Y.\,Yoshino \cite{IY}. 
In Section \ref{sec:model} we set the scene by recalling a recent
geometric model introduced by K.\,Baur and R.\,Marsh \cite{BMtube} and
by connecting it to another combinatorial model via arcs in an
$\infty$-gon, as introduced in \cite{HJ}.  
Section \ref{sec:classification} then contains the proofs of our 
main classification results. Finally, in Section \ref{sec:enumerate} 
we prove the enumeration results for torsion pairs in cluster tubes
and explain the cyclic sieving phenomenon.

\section{Torsion theory in triangulated categories}
\label{sec:torsion}

Torsion theory in abelian categories has been introduced
by S. Dickson \cite{Dickson} in 1966. More recently,
this concept has been taken to the world of triangulated 
categories by O. Iyama and Y. Yoshino \cite{IY}. In this section
we briefly review some fundamentals. 

For the rest of the paper we always assume that triangulated categories 
$\mathsf{C}$
are $K$-linear over a fixed field $K$, and
that they are Krull-Schmidt and
$\operatorname{Hom}$-finite (i.e.\ all morphism spaces are finite dimensional). 

Moreover, any subcategory of a triangulated category is always
supposed to be full, and closed under isomorphisms, direct sums
and direct summands.

\begin{Definition}
A {\em torsion pair} in a triangulated
category $\mathsf{C}$ is a pair $(\mathsf{X},\mathsf{Y})$ of  
subcategories such that
\begin{enumerate}
  \item[{(i)}] the morphism space $\operatorname{Hom}_{\mathsf{C}}(x,y)$ 
is zero for all $x \in \mathsf{X}$, $y \in \mathsf{Y}$,
  \item[{(ii)}] each object $c \in \mathsf{C}$ appears in a 
  distinguished triangle $x \rightarrow
  c \rightarrow y \rightarrow \Sigma x$ with $x \in \mathsf{X}$, $y \in \mathsf{Y}$.
\end{enumerate}
\end{Definition}

Examples of such torsion pairs in the triangulated situation 
are given by the t-structures of Beilinson, Bernstein, and Deligne \cite{BBD}
where, additionally, one assumes $\Sigma \mathsf{X} \subseteq \mathsf{X}$, 
and by the co-t-structures of Bondarko \cite{Bondarko} and Pauksztello
\cite{Pauksztello}
where, additionally, one assumes $\Sigma^{-1}\mathsf{X} \subseteq \mathsf{X}$.  
It is not hard to deduce from the definition that a torsion pair $(\mathsf{X},\mathsf{Y})$
is determined by one of its entries, namely one has that 
$$~\mbox{~~~\,\,\,\,\,~~~~~~~}
\mathsf{Y}=\mathsf{X}^{\perp}:=\{\, c \in \mathsf{C} \,|\, 
\operatorname{Hom}_{\mathsf{C}}(x,c) = 0 
\;\mbox{for each}\; x\in \mathsf{X} \,\},
$$
$$\mbox{and~~~}\mathsf{X}={}^{\perp}\mathsf{Y}:=\{\, c \in \mathsf{C} \,|\, 
\operatorname{Hom}_{\mathsf{C}}(c,y) = 0 
\;\mbox{for each}\; y \in \mathsf{Y} \,\}.$$ 
Moreover, as pointed out in \cite[rem.\ after def.\ 2.2]{IY},
for any torsion pair $(\mathsf{X},\mathsf{Y})$, both subcategories are
closed under extensions and $\mathsf{X}$ is contravariantly
finite in $\mathsf{C}$ and $\mathsf{Y}$ is covariantly finite 
in $\mathsf{C}$. 

Then we have the following useful characterisation of torsion
pairs, due to O.\,Iyama and Y.\,Yoshino. 

\begin{Proposition} (\cite[prop.\ 2.3]{IY})
\label{prop:IY-torsion}
Let $\mathsf{X}$ be a contravariantly finite
subcategory of $\mathsf{C}$. 
Then $(\mathsf{X},\mathsf{X}^{\perp})$ is a torsion
pair if and only if $\mathsf{X}={}^{\perp}(\mathsf{X}^{\perp})$.

Similarly, let $\mathsf{Y}$ be a covariantly finite subcategory 
of $\mathsf{C}$. Then $(\,^{\perp}\mathsf{Y},\mathsf{Y})$ 
is a torsion pair if and only 
if $\mathsf{Y}=({}^{\perp}\mathsf{Y})^{\perp}$.
\end{Proposition}

\section{A geometric model for cluster tubes}
\label{sec:model}

A geometric model for the cluster category associated to a tube
has been proposed by K. Baur and R. Marsh in \cite{BMtube}.
Instead of using their annulus model directly we prefer (as the authors
also partially do in \cite{BMtube}) to work 
with an infinite cover of it, namely arcs in an $\infty$-gon. 
The latter has vertices indexed by the integers; an arc
in the $\infty$-gon is a pair $(i,j)$ of integers with 
$j-i\ge 2$, and $i$ and $j$ are called endpoints of the arc. 

It is useful to think of an arc geometrically as
a curve between two integers on the number line as in the 
following picture. 
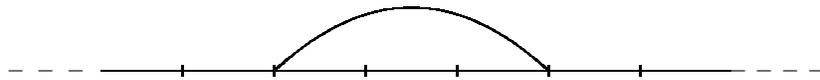
\[
  \xymatrix @-2.5pc @! {
    \rule{0ex}{6.5ex}\ar@{--}[r]&\ar@{-}[r]&*{\rule{0.1ex}{0.8ex}} \ar@{-}[r]& *{\rule{0.1ex}{0.8ex}} \ar@{-}[r] \ar@/^2pc/@{-}[rrr]& *{\rule{0.1ex}{0.8ex}} \ar@{-}[r] & *{\rule{0.1ex}{0.8ex}} \ar@{-}[r] & *{\rule{0.1ex}{0.8ex}} \ar@{-}[r] & *{\rule{0.1ex}{0.8ex}} \ar@{-}[r] &*{}\ar@{--}[r]&*{}
                    }
\]
This arc model has been extensively used in earlier papers,
see e.g. \cite{Ng}, \cite{HJ}; of particular importance 
is the situation where two arcs cross. 

\begin{Definition} \label{def:crossing}
Two arcs $(i,j)$ and $(k,l)$ of the $\infty$-gon
are said to {\em cross} if either
$i < k < j < l$ or $k < i < l < j$.   
\end{Definition}

This definition precisely corresponds to
the geometric intuition in that two arcs can be drawn as non-crossing
curves if and only if they do not cross in the sense of Definition 
\ref{def:crossing}. Note that two arcs which only
meet in some of their endpoints are not crossing. 

The Auslander-Reiten quivers of the tube category $\mathsf{T}_n$
of rank $n$ and the corresponding cluster tube $\mathsf{C}_{n}$
both have the shape of a cylinder of circumference 
$n$, infinite to one side. More precisely, it is of the form 
$\BZ A_{\infty}/(\tau^n)$ where $\tau$ is the Auslander-Reiten translation.
The indecomposable objects are arranged in the quiver as shown in
Figure 1, where the vertices are identified after $n$ steps. 
We will frequently use the standard coordinate system given in
Figure 1.
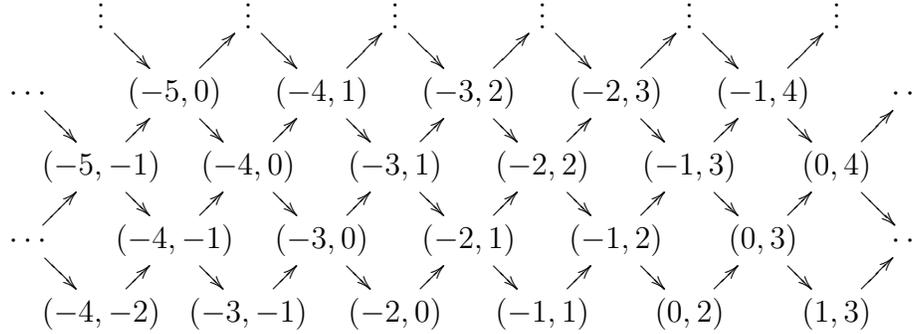
\begin{figure} \label{fig:coordinates}
\[
  \xymatrix @-4.0pc @! {
    & \vdots \ar[dr] & & \vdots \ar[dr] & & \vdots \ar[dr] & & \vdots \ar[dr] & & \vdots \ar[dr] & & \vdots & \\
    \cdots \ar[dr]& & (-5,0) \ar[ur] \ar[dr] & & (-4,1) \ar[ur] \ar[dr] & & (-3,2) \ar[ur] \ar[dr] & & (-2,3) \ar[ur] \ar[dr] & & (-1,4) \ar[ur] \ar[dr] & & \cdots \\
    & (-5,-1) \ar[ur] \ar[dr] & & (-4,0) \ar[ur] \ar[dr] & & (-3,1) \ar[ur] \ar[dr] & & (-2,2) \ar[ur] \ar[dr] & & (-1,3) \ar[ur] \ar[dr] & & (0,4) \ar[ur] \ar[dr] & \\
    \cdots \ar[ur]\ar[dr]& & (-4,-1) \ar[ur] \ar[dr] & & (-3,0) \ar[ur] \ar[dr] & & (-2,1) \ar[ur] \ar[dr] & & (-1,2) \ar[ur] \ar[dr] & & (0,3) \ar[ur] \ar[dr] & & \cdots\\
    & (-4,-2) \ar[ur] & & (-3,-1) \ar[ur] & & (-2,0) \ar[ur] & & (-1,1) \ar[ur] & & (0,2) \ar[ur] & & (1,3) \ar[ur] & \\
               }
\]
\caption{Coordinate system on the AR quiver}
\end{figure}

The Auslander-Reiten translation $\tau$ is on coordinates just given by 
$\tau:(i,j)\mapsto (i-1,j-1)$, i.e.\ for the AR-quiver of $\mathsf{T}_n$
and $\mathsf{C}_n$  
this implies that the coordinates $(i,j)$ and $(i-n,j-n)$ have to be 
identified for all $i,j$, i.e.\ both coordinates simultaneously have
to be taken modulo $n$. 

It is easy to see that there is a bijection between the vertices 
$(i,j)$ in the above coordinate system on the translation quiver $\mathbb{Z}A_{\infty}$ and arcs of the 
$\infty$-gon. 

There is a covering map from the translation quiver $\mathbb{Z}A_{\infty}$
to the AR-quiver of the cluster tube $\mathsf{C}_{n}$ which in terms
of the above coordinate system maps all vertices $(i+rn,j+rn)$, for 
$r\in \mathbb{Z}$, to the same vertex. To distinguish it notationally
we denote this vertex in the AR-quiver by $[(i,j)]$. In other words, 
every vertex $[(i,j)]$ in the AR-quiver of the cluster tube has 
infinitely many lifts $(i+rn,j+rn)$, where $r\in \mathbb{Z}$, 
to $\mathbb{Z}A_{\infty}$. In terms of arcs of the $\infty$-gon 
this means that a vertex $[(i,j)]$ in the AR-quiver of the cluster tube 
has to be identified with the collection of all arcs of the $\infty$-gon
of the form $(i+rn,j+rn)$, where $r\in \mathbb{Z}$.

Noting that full subcategories closed under direct sums and direct summands
are uniquely determined by the set of indecomposable objects they contain,
the following observation follows immediately.  
It will be used throughout the paper without 
further mentioning. 

\begin{Proposition}
Let $\mathsf{C}_n$ be the cluster category associated to a tube of rank $n$. 
Then there are bijections between the following sets: 
\begin{enumerate}
\item[{(i)}] Subcategories $\mathsf{X}$ of $\mathsf{C}_n$. 
\item[{(ii)}] Collections of arcs $\mathcal{X}$ of the $\infty$-gon
which are $n$-periodic, i.e.\ for each arc $(i,j)\in \mathcal{X}$ also
all arcs $(i+rn,j+rn)$ for $r\in \mathbb{Z}$
are in $\mathcal{X}$.
\end{enumerate}
\end{Proposition}

One of the main results in the paper by K.\,Baur and 
R.\,Marsh \cite{BMtube}
is a characterisation of the dimension of extension groups in the
categories $\mathsf{T}_n$ and $\mathsf{C}_n$ in terms of the arc model 
on the $\infty$-gon. For our purposes only the cluster tube situation
matters and we will now restate the relevant results from \cite{BMtube}.

For this, let us denote by $\sigma$ the map on the arcs of the 
$\infty$-gon mapping any arc to its shift by $n$ vertices , i.e.\ 
$\sigma((i,j))=(i+n,j+n)$, and hence $\sigma^r((i,j))=(i+rn,j+rn)$
for all $r\in \mathbb{Z}$.

\begin{Proposition} (\cite[thm.\ 3.8(c)]{BMtube}) \label{prop:BM}
Let $\alpha=[(i,j)]$ and $\beta=[(k,{\ell})]$ 
be indecomposable objects in the cluster tube
$\mathsf{C}_{n}$. Supposing w.l.o.g.\ that for the lengths of the arcs
we have ${\ell}-k\ge j-i$ we 
consider the following cardinalities,
$$I^+:= | \{m\in\mathbb{Z}\,|\,i < k+mn < j\} |
\mbox{~~and~~}
I^-:= | \{m\in\mathbb{Z}\,|\,i < {\ell}+mn < j\} |.
$$
Then we have
$$\dim \operatorname{Ext}^1_{\mathsf{C_n}}(\alpha,\beta) = I^+ + I^- =
\dim \operatorname{Ext}^1_{\mathsf{C_n}}(\beta,\alpha).
$$
\end{Proposition}

\begin{Remark} \label{rem:ext1}
\begin{enumerate}
\item[{(1)}] 
The cluster tube $\mathsf{C}_{n}$ is a 2-Calabi-Yau category and 
hence the functor $\operatorname{Ext}_{\mathsf{C}_n}^1$ is 
symmetric in the two arguments, up to vector space duality. 
In particular, the dimensions occurring in the above proposition 
agree and the assumption ${\ell}-k\ge j-i$ there  
(and in \cite[thm.\ 3.8(c)]{BMtube}) is not a restriction. 
\item[{(2)}] Note that $I^+ + I^-> 0$ if and only if some shifted arc 
$\sigma^m((k,{\ell}))$ of the $\infty$-gon crosses the arc $(i,j)$.
(This uses the assumption that ${\ell}-k\ge j-i$.)
\item[{(3)}] An indecomposable object $\alpha=[(i,j)]$ in 
$\mathsf{C}_n$ is rigid (i.e.\ 
$\operatorname{Ext}^1_{\mathsf{C}_n}(\alpha,\alpha)=0$) if and only if 
$j-i\le n$. In fact, this is exactly the condition for which no shift of 
$\alpha$ by a multiple of $n$ crosses $\alpha$. 
\end{enumerate}
\end{Remark}

We now characterise $\operatorname{Ext}^1$-vanishing in the cluster tube
$\mathsf{C}_n$ combinatorially
in terms of the arc model. 

For any collection $\mathcal{X}$ of arcs of the $\infty$-gon 
we denote by $\mathsf{nc}\,\mathcal{X}$ the collection of all arcs
of the $\infty$-gon which do not cross any arc from $\mathcal{X}$. 

\begin{Proposition} \label{prop:Ext1-nc} 
Let $\mathsf{X}$ be a subcategory of $\mathsf{C}_n$ and let 
$\mathcal{X}$ be the
corresponding $n$-periodic collection of arcs of the $\infty$-gon. For any indecomposable
object $y\in \mathsf{C}_n$ and corresponding set of arcs 
$\bar{y}$ of the 
$\infty$-gon the following are equivalent:
\begin{enumerate}
\item[{(i)}] 
$\operatorname{Ext}_{\mathsf{C}_n}^1(x,y)=0$ for all $x\in \mathsf{X}$.
\item[{(ii)}]
$\bar{y}\subseteq \mathsf{nc}\,\mathcal{X}.$
\end{enumerate}
In other words, the collection of arcs 
$\mathsf{nc}\,\mathcal{X}$ of the $\infty$-gon corresponds to the subcategory 
$\{y\in \mathsf{C}_n\,|\,\operatorname{Ext}_{\mathsf{C}_n}^1(x,y)=0
\mbox{~~for all $x\in \mathsf{X}$}\}$
of $\mathsf{C}_n$. 
\end{Proposition}

\begin{proof} First let $y\in \mathsf{C}_n$ be indecomposable such that 
$\operatorname{Ext}_{\mathsf{C}_n}^1(x,y)=0$ for all $x\in \mathsf{X}$, and let 
$\bar{y}$ be the corresponding set of arcs (all shifts of one of the arcs by 
multiples of $n$).  Then by Proposition \ref{prop:BM} and the remark
following it we have that
no arc from $\bar{y}$ crosses any shift of any arc from $\mathcal{X}$.
In particular, it does not cross any arc from $\mathcal{X}$ itself,
i.e.\ $\bar{y}\subseteq \mathsf{nc}\,\mathcal{X}$.

Conversely, let $\bar{y}\subseteq \mathsf{nc}\,\mathcal{X}$ be the 
set of arcs corresponding to an indecomposable object $y$ in $\mathsf{C}_n$. 
Take any indecomposable $x\in \mathsf{X}$ with corresponding set of arcs
$\bar{x}\subseteq \mathcal{X}$. 
Since $\bar{y}\subseteq \mathsf{nc}\,\mathcal{X}$ by assumption,
no arc in $\bar{y}$ crosses any arc in $\bar{x}$. Note that both
$\bar{x}$ and $\bar{y}$ are $n$-periodic by definition, so 
Proposition \ref{prop:BM} and Remark \ref{rem:ext1}
imply that $\operatorname{Ext}_{\mathsf{C}_n}^1(x,y)=0$, as claimed.
\end{proof}

We recall the following definition from \cite{Ng}; see also \cite{HJR-Ptolemy} for a finite version, which will be recalled 
in Section \ref{sec:enumerate} below.

\begin{Definition} 
\label{def:Ptolemy}
A collection $\mathcal{X}$ of arcs of the $\infty$-gon is 
called a {\em Ptolemy diagram} if the following condition is satisfied:
if $(i,j)$ and $(r,s)$ are arcs in $\mathcal{X}$ which cross,
w.l.o.g.\ $i<r$, then
all arcs among the pairs $(i,r)$, $(i,s)$, $(r,j)$, $(j,s)$ 
must also be in $\mathcal{X}$. 
\end{Definition}

\begin{Remark} \label{rem:nc-Ptolemy}
For any collection $\mathcal{X}$ of arcs of the $\infty$-gon,
$\mathsf{nc}\,\mathcal{X}$ is a Ptolemy diagram. 
In fact, suppose $(i,j)$ and $(r,s)$ are crossing 
arcs in $\mathsf{nc}\,\mathcal{X}$ with $i<r$. 
We have to show that the arcs
$(i,r)$, $(i,s)$, $(r,j)$, $(j,s)$ are also in $\mathsf{nc}\,\mathcal{X}$.
This immediately follows from the observation that any arc of the
$\infty$-gon crossing one of $(i,r)$, $(i,s)$, $(r,j)$, $(j,s)$
must cross at least one of $(i,j)$ and $(r,s)$. Since $(i,j)$ and 
$(r,s)$ are in $\mathsf{nc}\,\mathcal{X}$ by assumption this means 
that no arc from $\mathcal{X}$ can cross either of
$(i,r)$, $(i,s)$, $(r,j)$, $(j,s)$, i.e. these are in 
$\mathsf{nc}\,\mathcal{X}$, as claimed.
\end{Remark}

\begin{Proposition} \label{prop:Ptolemy}
Let $\mathsf{X}$ be a subcategory of the cluster tube $\mathsf{C}_n$, 
and let $\mathcal{X}$
be the corresponding $n$-periodic collection of arcs of the $\infty$-gon. 
Then the following conditions are equivalent:
\begin{enumerate}
\item[{(i)}] $\mathsf{X}=\,^{\perp}(\mathsf{X}^{\perp})$; 
\item[{(ii)}] $\mathcal{X} = \mathsf{nc}\,\mathsf{nc}\,\mathcal{X}$;
\item[{(iii)}] $\mathcal{X}$ is an $n$-periodic Ptolemy diagram.
\end{enumerate}
\end{Proposition}

\begin{proof}
For the equivalence of (i) and (ii) we observe the following. By 
Proposition \ref{prop:Ext1-nc} the subcategory 
\begin{eqnarray*}
\mathsf{X}^{\perp} & = & \{y\in\mathsf{C}_n\,|\,
\operatorname{Hom}_{\mathsf{C}_n}(x,y)=0
\mbox{~~for all $x\in \mathsf{X}$}\} \\
& = & \{y\in\mathsf{C}_n\,|\,
\operatorname{Ext}_{\mathsf{C}_n}^1(x,\Sigma^{-1}y)=0
\mbox{~~for all $x\in \mathsf{X}$}\} 
\end{eqnarray*}
corresponds to the collection of arcs $\Sigma\, \mathsf{nc}\,\mathcal{X}$.
Similarly, $~^{\perp}\mathsf{X}$ corresponds to 
$\Sigma^{-1}\,\mathsf{nc}\,\mathcal{X}$.

Taken together we have 
$\mathsf{X}=\,^{\perp}(\mathsf{X}^{\perp})$ if and only if 
$$\mathcal{X} = \Sigma^{-1}\, \mathsf{nc}\,
(\Sigma\, \mathsf{nc}\,\mathcal{X}) = \mathsf{nc}\,\mathsf{nc}\,\mathcal{X}
$$
where the latter equation holds because $\Sigma$ and $\mathsf{nc}$
commute (since $\mathsf{C}_n$ is 2-Calabi-Yau, the suspension 
$\Sigma$ equals the AR translation $\tau$ and is thus just a shift by 1 
on the AR quiver). 

The implication (ii) $\Longrightarrow$ (iii) immediately follows 
from \cite[lem.\ 3.17]{Ng}; one just has to observe that condition 
(i) in \cite[def.\ 0.3]{Ng} is precisely our Ptolemy condition
as given in Definition \ref{def:Ptolemy}.

Conversely, suppose (iii) holds, i.e.\ $\mathcal{X}$ is an
$n$-periodic Ptolemy diagram. For deducing that 
$\mathcal{X}=\mathsf{nc}\,\mathsf{nc}\,\mathcal{X}$
we can again use \cite[lem.\ 3.17]{Ng}; hence we have to show 
that condition (ii) of \cite[def.\ 0.4]{Ng} does not occur
in our context. In fact, let a vertex $i$ be a left fountain
and a vertex $j$ with $j-i\ge 2$ be a right fountain of $\mathcal{X}$, 
i.e.\ there are infinitely many arcs in $\mathcal{X}$ of the form 
$(k,i)$ and infinitely many arcs in $\mathcal{X}$
of the form $(j,{\ell})$. 
Among the former pick an arc $(k,i)$ with $i-k\ge n+1$. By $n$-periodicity
all shifts by multiples of $n$ of the infinitely many arcs of the 
form $(j,{\ell})$ are also in $\mathcal{X}$. Since $i-k\ge n+1$, there are 
infinitely many of the arcs $(j,{\ell})$ for which some shift by a multiple
of $n$ crosses
the arc $(k,i)$. But since $\mathcal{X}$ is a Ptolemy diagram by 
assumption, we conclude that $i$ must also be a right fountain. 
This shows that condition (ii) of \cite[def.\ 0.4]{Ng} is obsolete
and hence the implication (iii) $\Longrightarrow$ (ii) in our
proposition also 
follows from \cite[lem.\ 3.17]{Ng}. 
\end{proof}

Recall from Figure 1 the coordinate system on 
the AR quiver 
of $\mathsf{C}_n$. If a vertex of the AR quiver has coordinates
$[(i,j)]$ we call the natural number $j-i-1$ the {\em level} of the 
vertex. Note that the mouth of the cylinder (i.e. the bottom line
of the AR quiver) contains the vertices
of level 1.

\begin{Proposition} \label{prop:finite-infinite}
Let $\mathsf{X}$ be a subcategory of the cluster tube $\mathsf{C}_n$, 
and let $\mathcal{X}$
be the corresponding $n$-periodic collection of arcs of the $\infty$-gon. 

Suppose that $\mathsf{X}=\,^{\perp}(\mathsf{X}^{\perp})$, or
equivalently that $\mathcal{X}$ is an $n$-periodic Ptolemy diagram.

Then precisely one of the following situations occurs:
\begin{enumerate}
\item[{(i)}] $\mathsf{X}$ has only finitely many indecomposable objects, all
of level $\le n-1$,
and $\mathsf{X}^{\perp}$ contains infinitely many indecomposable objects.
\item[{(ii)}] $\mathsf{X}$ has infinitely many indecomposable objects,
and $\mathsf{X}^{\perp}$ contains only finitely many indecomposable objects,
all of level $\le n-1$. 
\end{enumerate}
\end{Proposition}

\begin{proof} By Proposition \ref{prop:Ptolemy} the condition
$\mathsf{X}=\,^{\perp}(\mathsf{X}^{\perp})$ is indeed equivalent
to $\mathcal{X}$ being an $n$-periodic Ptolemy diagram.
So we can 
prove the proposition in terms of the Ptolemy diagrams $\mathcal{X}$
and $\Sigma\,\mathsf{nc}\,\mathcal{X}$ associated 
with the subcategories $\mathsf{X}$ and $\mathsf{X}^{\perp}$, respectively. 

(1) We first show that if
$\mathsf{X}$ (or $\mathsf{X}^{\perp}$)
contains an indecomposable object of level $\ge n$
then $\mathsf{X}$ (or $\mathsf{X}^{\perp}$)
contains infinitely many indecomposable objects. 

In fact, by assumption there is an arc in $\mathcal{X}$ of the form 
$(i,i+kn+{\ell})$ with $k\in\mathbb{N}$ and $1\le {\ell} < n$.
By periodicity the arc $(i+kn,i+2kn+{\ell})$ 
is also in $\mathcal{X}$ and it crosses $(i,i+kn+{\ell})$ because 
$1\le {\ell}$. 
By assumption, $\mathcal{X}$ is a Ptolemy diagram which 
implies that the arc $(i,i+2kn+{\ell})$ must then be in $\mathcal{X}$ as well. 
Now repeat the argument starting with the arc 
$(i,i+2kn+{\ell})\in\mathcal{X}$;
inductively this gives infinitely many arcs in $\mathcal{X}$. 
Having pairwise different lengths, no two of them are shifts by multiples 
of $n$ of each other, i.e. they represent infinitely many different 
indecomposable objects in $\mathsf{X}$. 

The above argument only used that $\mathcal{X}$ is a Ptolemy diagram,
so it applies verbatim also to the Ptolemy diagram 
$\Sigma\,\mathsf{nc}\,\mathcal{X}$ (cf. Remark \ref{rem:nc-Ptolemy})
and the corresponding subcategory $\mathsf{X}^{\perp}$. 

Thus we have already shown that if one $\mathsf{X}$ and $\mathsf{X}^{\perp}$
has only finitely many indecomposable objects then these objects 
are all located strictly below level $n$ in the AR quiver of 
$\mathsf{C}_n$. 
 
(2) We now show that at least one of the subcategories 
$\mathsf{X}$ and $\mathsf{X}^{\perp}$ has
infinitely many indecomposable objects, i.e.\ we claim that if 
$\mathsf{X}$ (or $\mathsf{X}^{\perp}$) has only finitely many 
indecomposable objects then $\mathsf{X}^{\perp}$ (or $\mathsf{X}$)
has infinitely many indecomposable objects. 

By Proposition \ref{prop:Ptolemy} this amounts to showing for the 
corresponding collections of arcs that if $\mathcal{X}$ is finite
then $\Sigma\,\mathsf{nc}\,\mathcal{X}$ contains arcs of arbitrary length. 
Since $\Sigma$ only shifts by one vertex, the latter is equivalent
to $\mathsf{nc}\,\mathcal{X}$ containing arcs of arbitrary length. 

Pick an arc from the finite collection $\mathcal{X}$ of maximal length. 
W.l.o.g.\ (up to shifting) suppose its coordinates are $(0,{\ell})$ with 
${\ell}\ge 2$. The vertex ${\ell}$ of the $\infty$-gon
can not be 'overarched' by an arc from 
$\mathcal{X}$, i.e. there is no arc $(i,j)\in\mathcal{X}$ with
$i<{\ell}<j$; in fact, because $(0,{\ell})$ has maximal length, an
arc overarching the vertex ${\ell}$ 
had to cross $(0,{\ell})$, 
but then the Ptolemy condition would yield an arc
in $\mathcal{X}$ longer than $(0,{\ell})$, a contradiction.

Since the collection $\mathcal{X}$ is $n$-periodic this even implies that
no vertex of the $\infty$-gon
of the form ${\ell}+rn$ with $r\in\mathbb{Z}$ can be overarched
by an arc from $\mathcal{X}$. 

But this means that all the arcs $({\ell}+rn,{\ell}+sn)$ for integers $r<s$
can not cross any arc from $\mathcal{X}$, i.e.\ are in 
$\mathsf{nc}\,\mathcal{X}$. In particular, $\mathsf{nc}\,\mathcal{X}$
contains infinitely many arcs. 

Again, the above argument only used that $\mathcal{X}$ is a Ptolemy 
diagram and therefore applies verbatim also to 
the Ptolemy diagram 
$\Sigma\,\mathsf{nc}\,\mathcal{X}$, i.e.\ to $\mathsf{X}^{\perp}$. 

(3) For completing the proof of Proposition \ref{prop:finite-infinite}
it remains to show that not both subcategories $\mathsf{X}$ and
$\mathsf{X}^{\perp}$ can have 
infinitely many indecomposable objects. 

Suppose that $\mathsf{X}$ has infinitely many indecomposable objects.
In particular, 
in the corresponding collection $\mathcal{X}$ of arcs there is
an arc of some length ${\ell}\ge n+1$; w.l.o.g.\ (shifting) say its 
coordinates are $(0,{\ell})$. Now take any arc of the $\infty$-gon
of length $\ge {\ell}$. This can not be in
$\mathsf{nc}\,\mathcal{X}$ since by $n$-periodicity of $\mathcal{X}$
a shift of $(0,{\ell})$  
would cross the given arc since ${\ell}\ge n+1$. This argument says that as 
soon as $\mathcal{X}$ has an arc of length ${\ell}\ge n+1$, all arcs of length
$\ge {\ell}$ can't be in $\mathsf{nc}\,\mathcal{X}$. After identifying 
any arc in $\mathsf{nc}\,\mathcal{X}$ with all its shifts by multiples of 
$n$ this means that there are only finitely many indecomposable objects 
in the corresponding subcategory $\Sigma^{-1}\,\mathsf{X}^{\perp}$,
and hence also in $\mathsf{X}^{\perp}$, as claimed. 

Again, the argument works equally well with the roles of $\mathsf{X}$
and $\mathsf{X}^{\perp}$ interchanged. 

This completes the proof of Proposition \ref{prop:finite-infinite}.
\end{proof}

\section{The classification of torsion pairs}
\label{sec:classification}

In this section we will classify the torsion pairs
in the cluster tube $\mathsf{C}_n$, based on the arc model via 
$n$-periodic Ptolemy diagrams (cf. Propositions \ref{prop:Ptolemy}
and \ref{prop:finite-infinite}).

Recall from Proposition \ref{prop:IY-torsion} (and the remarks
preceding it) that every torsion pair 
in $\mathsf{C}_n$ has the form $(\mathsf{X},\mathsf{X}^{\perp})$
where $\mathsf{X}$ is a contravariantly finite 
subcategory.

\begin{Proposition} 
\label{prop:torsion-vs-Ptolemy}
\begin{enumerate}
\item[{(a)}] Let $(\mathsf{X},\mathsf{X}^{\perp})$ be a torsion
pair in the cluster tube $\mathsf{C}_n$. Then the corresponding
$n$-periodic collections of arcs $\mathcal{X}$ and 
$\Sigma\,\mathsf{nc}\,\mathcal{X}$ of the $\infty$-gon
are both Ptolemy diagrams
and precisely one of the subcategories $\mathsf{X}$ and 
$\mathsf{X}^{\perp}$ has only finitely many indecomposable objects. 
\item[{(b)}] Let $\mathcal{X}$ be an $n$-periodic Ptolemy diagram
of the $\infty$-gon and let $\mathsf{X}$ be the corresponding 
subcategory of $\mathsf{C}_n$. Then precisely one of the subcategories 
$\mathsf{X}$ and $\mathsf{X}^{\perp}$ has only finitely many 
indecomposable objects, and $(\mathsf{X},\mathsf{X}^{\perp})$
is a torsion pair.    
\end{enumerate}
\end{Proposition}

\begin{proof} 
(a) Let $(\mathsf{X},\mathsf{X}^{\perp})$ be a torsion pair.
By Proposition \ref{prop:IY-torsion} we have
$\mathsf{X} =\,^{\perp}(\mathsf{X}^{\perp})$ and then
by Proposition \ref{prop:Ptolemy}
the corresponding set of arcs $\mathcal{X}$ is a Ptolemy diagram. 
Moreover, by Remark \ref{rem:nc-Ptolemy} we have that also
$\Sigma\,\mathsf{nc}\,\mathcal{X} = \mathsf{nc}\,\Sigma\,\mathcal{X}$
is an $n$-periodic Ptolemy diagram. Finally, 
Proposition \ref{prop:finite-infinite} says that 
$\mathsf{X} =\,^{\perp}(\mathsf{X}^{\perp})$ implies that
precisely one of the subcategories $\mathsf{X}$ and 
$\mathsf{X}^{\perp}$ has only finitely many indecomposable objects. 

(b) Let $\mathcal{X}$ be an $n$-periodic Ptolemy diagram of the
$\infty$-gon. Proposition \ref{prop:finite-infinite} immediately
gives that precisely one of the subcategories $\mathsf{X}$ and 
$\mathsf{X}^{\perp}$ has only finitely many indecomposable objects.
It remains to show that $(\mathsf{X},\mathsf{X}^{\perp})$ is indeed
a torsion pair. If $\mathsf{X}$ (resp. $\mathsf{X}^{\perp}$) is
the subcategory with only finitely many indecomposable objects then
it is clearly functorially finite, thus in particular
contravariantly finite (resp. covariantly finite). 
The corresponding collections of arcs $\mathcal{X}$ (resp.
$\Sigma\,\mathsf{nc}\,\mathcal{X}$) are Ptolemy diagrams by
assumption (resp. by Remark \ref{rem:nc-Ptolemy}). Applying 
Proposition \ref{prop:Ptolemy} we deduce that 
$\mathsf{X}=\,^{\perp}(\mathsf{X}^{\perp})$ 
(resp. 
$\mathsf{X}^{\perp} = (\,^{\perp}(\mathsf{X}^{\perp}))^{\perp}$).
Then Proposition \ref{prop:IY-torsion} implies that
$(\mathsf{X},\mathsf{X}^{\perp})$ (resp. 
$(\,^{\perp}(\mathsf{X}^{\perp}),\mathsf{X}^{\perp})$)
is a torsion pair. Since $\mathsf{X}=\,^{\perp}(\mathsf{X}^{\perp})$
we obtain in any case that 
$(\mathsf{X},\mathsf{X}^{\perp})$ is a torsion pair, as claimed. 
\end{proof}

As an immediate 
application we can give an alternative proof of the following result 
of A.\,Buan, R.\,Marsh and D.\,Vatne.

\begin{Corollary}[\cite{BMV}]
The cluster tubes $\mathsf{C}_n$ do not contain any cluster tilting
objects. 
\end{Corollary}

\begin{proof} If $u$ was a cluster tilting object in $\mathsf{C}_n$
then $(\mathsf{add}(u),\Sigma\,\mathsf{add}(u))$ would be a cluster
tilting subcategory (cf. \cite[sec.\ 2.1]{KR}). In particular,
it would be a torsion pair, but one for which both subcategories
$\mathsf{add}(u)$ and $\Sigma\,\mathsf{add}(u)$ would have
finitely many indecomposable objects, contradicting Proposition
\ref{prop:torsion-vs-Ptolemy}.
\end{proof}

By Proposition \ref{prop:torsion-vs-Ptolemy} the classification of
torsion pairs $(\mathsf{X},\mathsf{X}^{\perp})$ in the cluster tube
$\mathsf{C}_n$ reduces to the classification of the possible halves
$\mathsf{X}$ (or $\mathsf{X}^{\perp}$) of a torsion pair containing
only finitely many indecomposable objects. We also know from
Proposition~\ref{prop:finite-infinite} that in this case all
indecomposable objects are strictly below level $n$ in the AR-quiver
of $\mathsf{C}_n$.

We shall need the following definition.

\begin{Definition}
  Let $(i,j)$ be an arc of the $\infty$-gon.  The {\em wing} $W(i,j)$
  of $(i,j)$ consists of all arcs $(r,s)$ of the $\infty$-gon such
  that $i\le r\le s\le j$, i.e.\ all arcs which are {\em overarched}
  by $(i,j)$.

  Interpreting $(i,j)$ as a vertex $[(i,j)]$ in the AR-quiver of the
  cluster tube $\mathsf{C}_n$, the corresponding wing is denoted
  $W[(i,j)]$; it consists of all vertices in the region bounded by 
  the lines from $[(i,j)]$ down to the vertices $[(i,i+2)]$ and 
  $[(j-2,j)]$,
  respectively.   
 
\end{Definition}

Then we are in the position to prove our main classification
result for torsion pairs in cluster tubes.

\begin{Theorem} \label{thm:torsion}
There are bijections between the following sets:
\begin{enumerate}
\item[{(i)}] 
Torsion pairs $(\mathsf{X},\mathsf{X}^{\perp})$
in the cluster tube $\mathsf{C}_n$ such that $\mathsf{X}$ has only
finitely many indecomposable objects. 
\item[{(ii)}] $n$-periodic Ptolemy diagrams $\mathcal{X}$ of the
$\infty$-gon such that all arcs in $\mathcal{X}$ have length at most $n$.
\item[{(iii)}] Collections 
$\{([(i_1,j_1)],[W_1]),\ldots,([(i_r,j_r)],[W_r])\}$
of pairs consisting of vertices $[(i_{\ell},j_{\ell})]$
of level $\le n-1$ in the AR-quiver of $\mathsf{C}_n$ and 
subsets $[W_{\ell}]\subseteq W[(i_{\ell},j_{\ell})]$ of their wings 
such that for any different $k,{\ell}\in\{1,\ldots,r\}$ we have
$$\Sigma\, W[(i_k,j_k)]\cap W[(i_{\ell},j_{\ell})]
=\emptyset
$$
and the $n$-periodic collection $W_{\ell}$ of arcs corresponding to $[W_{\ell}]$ is a Ptolemy diagram (in which every arc is 
overarched by some arc from the collection corresponding 
to $[(i_{\ell},j_{\ell})]$). 
\end{enumerate}
\end{Theorem}

\begin{proof}
By Proposition \ref{prop:Ptolemy} we have to determine the
possible 
$n$-periodic Ptolemy diagrams $\mathcal{X}$ corresponding to
a subcategory $\mathsf{X}$ which appears as the first
half of a torsion pair $(\mathsf{X},\mathsf{X}^{\perp})$ in
$\mathsf{C}_n$ and has only finitely many indecomposable
objects.

Clearly, for such a Ptolemy diagram $\mathcal{X}$ all arcs
must have length at most $n$ (recall that the level of a vertex 
$[(i,j)]$ is defined as $j-i-1$, and the corresponding arc $(i,j)$
has length $j-i$). 

Conversely, let $\mathcal{X}$ be an $n$-periodic Ptolemy diagram
such that all arcs have length at most $n$. 
We have to show that the
corresponding subcategory $\mathsf{X}$ is the finite half of a
torsion pair. Coming from a Ptolemy diagram it is immediate from
Proposition \ref{prop:torsion-vs-Ptolemy} that $\mathsf{X}$ is one half of a 
torsion pair. Moreover, the bounded arc lengths clearly force  
$\mathsf{X}$ to be the subcategory having only finitely many 
indecomposable objects (in fact, an arc of length $m$
corresponds to an indecomposable object of level $m-1$, and
for each $m$ there are only finitely many such indecomposable
objects). 

This shows that the condition (ii) classifies finite halves of
torsion pairs.

For obtaining the description given in (iii) 
one just has to translate the arc
picture from (ii) to the AR-quiver of $\mathsf{C}_n$. 
The vertices $[(i_{\ell},j_{\ell})]$ occurring in (iii) 
are the ones corresponding to those 
arcs $(i_{\ell},j_{\ell})$ in $\mathcal{X}$
which are not overarched by any other arc from $\mathcal{X}$.  
Then, using the 
coordinate system on the AR-quiver one observes
that the condition 
$\Sigma\, W[(i_k,j_k)]\cap W[(i_{\ell},j_{\ell})]=\emptyset$ 
for any different $k,{\ell}$ just expresses the property
that the corresponding arcs $(i_k,j_k)$ and 
$(i_{\ell},j_{\ell})$ do not cross (since they are not
overarched by any arc from $\mathcal{X}$ they must not 
cross by the Ptolemy condition). 
The fact that the collections $W_{\ell}$ of arcs corresponding 
to $[W_{\ell}]$
can be chosen as arbitrary $n$-periodic Ptolemy diagrams
also has been observed above.  
\end{proof}

\section{Enumeration of torsion pairs}
\label{sec:enumerate}

For the enumeration of all torsion pairs in $\mathsf{C}_n$ it is
convenient to rephrase the characterisation of the collections in
Theorem~\ref{thm:torsion} (iii) as follows: torsion pairs
$(\mathsf{X},\mathsf{X}^{\perp})$ in the cluster tube $\mathsf{C}_n$
such that $\mathsf{X}$ has only finitely many indecomposable objects
are in bijection with
\begin{enumerate}
\item[{(iv)}] Collections $\{([(i_1,i_2)],[W_1]),
  ([(i_2,i_3)],[W_2]),\ldots,([(i_r,i_{r+1})],[W_r])\}$ with $1\leq
  i_{k+1}-i_k \leq n$ and $i_{r+1} = i_1+n$, where $[W_k]$ is the empty
  set if $i_{k+1}-i_k = 1$ and $[W_k]$ is a subset of the wing
  $W[(i_k, i_{k+1})]$ otherwise.  In the latter case, $[W_k]$
  interpreted as an $n$-periodic collection of arcs is required to be
  a Ptolemy
  diagram in which every arc is overarched by some arc from the
  collection corresponding to $[(i_k, i_{k+1})]$.
\end{enumerate}
Note that we have $\Sigma\, W[(i_k,i_{k+1})]\cap
W[(i_\ell,i_{\ell+1})] =\emptyset$ for all $k\neq\ell$ and whenever
$i_{k+1}-i_k\geq 2$ and $i_{\ell+1}-i_\ell\geq 2$ since the arcs
corresponding to the pairs in the collections
$[(i_1,i_2)],[(i_2,i_3)],\dots,[(i_r,i_1+n)]$ do not cross.

To proceed we need a combinatorial description of the individual
pairs $([(i,j)], [W])$ appearing in the characterisation (iv) above.
These pairs can be interpreted as sets of diagonals of an
$(j-i+1)$-gon with vertices labelled clockwise from $i$ to $j$ and
where a diagonal connecting vertices $k$ and $\ell$ is present if and
only if the arc $(k, \ell)$ is in $W$.  In the following, the edge
between vertices $i$ and $j$ is called the \Dfn{distinguished base
  edge}.

With this interpretation, the pairs appearing in the characterisation
(iv) above are precisely the Ptolemy diagrams on the $(j-i+1)$-gon
studied in~\cite{HJR-Ptolemy} and~\cite{Kluge-Rubey}, except that
there the vertices are not labelled.  As shown in~\cite[prop.\
2.5]{HJR-Ptolemy} (see also~\cite[sec.\ 1.1]{Kluge-Rubey}) the set
$\P$ of all such (unlabelled) Ptolemy diagrams with distinguished
base edge is the disjoint union of%
\begin{itemize}
\item the degenerate Ptolemy diagram, consisting of two vertices
  and the distinguished base edge only,
\item a triangle with a distinguished base edge and two Ptolemy
  diagrams glued along their distinguished base edges onto the
  other edges,
\item a clique, i.e.\ a diagram with at least four edges and
  all diagonals present, with a distinguished base edge and Ptolemy
  diagrams glued along their distinguished base edges onto the
  other edges,
\item an empty cell, i.e.\ a polygon with at least four edges
  without diagonals, with a distinguished base edge and Ptolemy
  diagrams glued along their distinguished base edges onto the
  other edges.
\end{itemize}
We call the vertex coming first on the base edge when going
counterclockwise the \Dfn{distinguished base vertex}.

As in~\cite[sec.\ 3.b.]{HJR-Ptolemy} we now apply the theory of
combinatorial species together with Lagrange inversion to find a
formula for the number of torsion pairs in the cluster tube
$\mathsf{C}_n$.  The necessary background can be found there, or, in
more detail, in the book by Bergeron, Labelle and Leroux~\cite[sec.\
1.3]{BLL}.

We turn $\P$ into a combinatorial species as follows: for any set $U$
of cardinality $n$ let $\P[U]$ be the set of all Ptolemy diagrams on
the $(n+1)$-gon whose vertices other than the base vertex are
labelled (bijectively) with the elements of $U$.  Let $\P(z) =
\P(z,x,y_1,y_2)$ be the associated weighted exponential generating
function, where the exponent of $x$ (of $y_1$, of $y_2$) records the
number of triangles (cliques, empty cells respectively).

\begin{Remark}\label{rmk:ordinary-gf}
  Note that there are $n!$ ways to label the vertices other than the
  base vertex, thus the exponential generating function $\P(z)$
  coincides with its so-called \lq ordinary\rq\ generating function:
  \begin{align*}
    \P(z)&=\sum_{n\geq 1}\sum_{P \in\P[\{1,\dots,n\}]}
    x^{\#\text{triangles in $P$}} y_1^{\#\text{cliques in $P$}}
    y_2^{\#\text{empty cells in $P$}} %
    \,\,\frac{z^n}{n!}\\
    &=\sum_{n\geq 1}\sum_{\substack{P \text{ a Ptolemy diagram}\\\text{on the
        $(n+1)$-gon}}}%
    x^{\#\text{triangles in $P$}} y_1^{\#\text{cliques in $P$}}
    y_2^{\#\text{empty cells in $P$}} %
    \,z^n.
  \end{align*}
\end{Remark}
As we include the degenerate diagram consisting of two vertices we
have 
$$\P(z) = z + x z^2 + (2x^2 + y_1 + y_2) z^3 +\dots$$

Let $C$ be the \Dfn{species of cycles} which for any set
$U=\{u_1,\dots,u_n\}$ produces the set $C[U]$ of all (oriented)
cycles with $n$ vertices labelled $u_1,\dots,u_n$.  There are
$(n-1)!$ such labelled cycles, thus the associated exponential
generating function is
$$
C(z) = \sum_{n\geq 0} \# C[\{1,2,\dots,n\}] \frac{z^n}{n!} =
\sum_{n\geq 1} \frac{z^n}{n} = \log\left(\frac{1}{1-z}\right).
$$

We also need the following two operations on species: $\F\circ\G$
denotes the \Dfn{composition} of two species $\F$ and $\G$, with
associated exponential generating function $\F\big(\G(z)\big)$.
Intuitively, the set $(\F\circ\G)[U]$ can be visualised by taking an
object produced by $\F$, and replacing all its labels by objects
produced by $\G$, such that the set of labels is exactly $U$.  A
precise definition is given in \cite{HJR-Ptolemy} and \cite[sec.\
1]{BLL}.

The second operation is called \Dfn{pointing}: for a set of labels
$U$ let $\F^\bullet[U] = \{(f, u): f\in\F[U], u\in U\}$.  In other
words, $\F^\bullet$ produces all $\F$-structures with a distinguished
label.  The associated exponential generating function satisfies
$\F^\bullet(z) = z\F'(z)$, where $\F'(z)$ is the derivative of
$\F(z)$.

\begin{Lemma}\label{lem:bij-tors-pairs}  
  Let $U$ be a set of cardinality $n$.  Then $(C\circ \P)^\bullet[U]$
  is in bijection with the set of $n$-periodic Ptolemy diagrams of
  the $\infty$-gon with all arcs having length at most $n$, where the
  cosets of the vertices of the $\infty$-gon modulo $n$ are labelled
  (bijectively) with the elements of $U$.
\end{Lemma}
\begin{proof}
  Let $\left((P_1, P_2, \dots, P_r), u\right)$ be an element of
  $(C\circ \P)^\bullet[U]$.  Let $s$ be such that $u$ is a label of
  $P_s$ and suppose that $u$ labels the $\ell$\textsuperscript{th}
  vertex of $P_s$ when going clockwise along the border of the
  diagram, starting after the base vertex.  Thus, if $P_s$ is a
  Ptolemy diagram on the $(n_s+1)$-gon, we have $1\leq\ell\leq n_s$.

  We can now construct in a bijective way a collection of pairs as in
  the characterisation (iv) given above.  First we set $i_1=-\ell$
  and $i_{k+1} = i_k + n_k$ for $1\leq k\leq r$.

  We then obtain $[W_{k+1}]$ from $P_{k+s}$ for $k\geq 0$ (setting
  $P_{k+s} = P_{k+s-r}$ for $k+s>r$) by the following procedure.
  Beginning with $k=0$, going clockwise along the border of the
  diagram $P_s$, starting at the base vertex, we number the vertices
  from $-\ell$ to $n_s-\ell$.  In particular, the vertex labelled
  with $u$ is assigned the number $0$.

  Having numbered the vertices of $P_{k+s}$, the last number used is
  also assigned to the base vertex of $P_{k+s+1}$.  We continue the
  numbering in this fashion.

  The indecomposable object with coordinates $[(i,j)]$, $j-i\geq 2$ is included in
  $[W_{k+1}]$ if and only if there is an arc from the vertex numbered
  $i$ to the vertex numbered $j$ in $P_{k+s}$.

  Finally we label the coset of vertices of the $\infty$-gon
  containing $i$ with the label of the unique non-base vertex
  numbered $i$ in one of the $P_k$.
\end{proof}

\begin{Example}
  Let us give an example for this bijection when $n=10$.  Consider
  the collection of pairs
  \begin{align*}
    \big\{
    &\big((-2,2), \{(8,1),(8,2),(9,1)\}\big),\\
    &\big((2,3), \{\}\big),\\
    &\big((3,6), \{(3,5),(3,6),(4,6)\}\big),\\
    &\big((6,8), \{(6,8)\}\big)\big\}.
  \end{align*}
  This collection satisfies the conditions in characterisation~(iv).

  The finite half of the corresponding torsion pair is visualised
  below.  In this picture the vertices of the first three levels of
  the quiver are shown (coordinates abbreviated modulo $n$) and
  arrows are omitted.  The vertices corresponding to the
  indecomposable objects are boxed and the wings are indicated by dotted
  lines (extended below level~$1$ to improve visibility).%
  
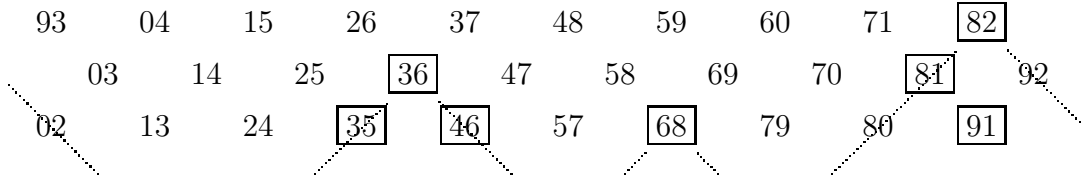
\[
  \xymatrix @-2.5pc @! {
    &93&&04&&15&&26&&37&&48&&59&&60&&71&&\fbox{82}\ar@{.}[dddlll]\ar@{.}[ddrr]&&\\
    \ar@{.}[ddrr]&&03&&14&&25&&\fbox{36}\ar@{.}[ddll]\ar@{.}[ddrr]&&47&&58&&69&&70&&\fbox{81}&&92&&\\
    &02&&13&&24&&\fbox{35}&&\fbox{46}&&57&&\fbox{68}\ar@{.}[dl]\ar@{.}[dr]&&79&&80&&\fbox{91}&&\\
    &&&&&&&&&&&&&&&&&&&&&&&&&&\\
  }
  \]
  Labelling vertex $i$ of the $\infty$-gon with the label
  $i\bmod{10}$, the corresponding cycle of Ptolemy diagrams is
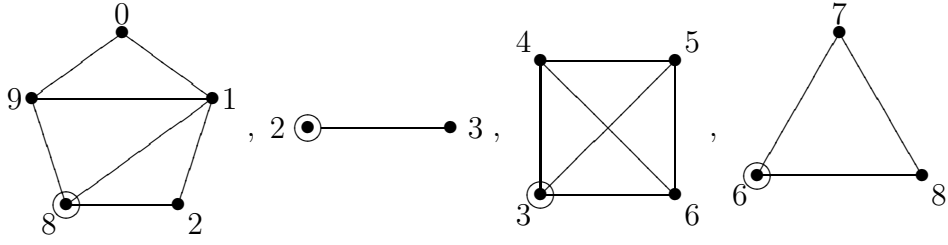
\[
\xy/r3pc/: 
{\xypolygon5"A"{\bullet}},
"A1"!{+R*+!L{1}},
"A2"!{+UU*+!DD{0}},
"A3"!{+L*+!R{9}},
"A4"!{+LD*+!RU{8}*\cir<5pt>{}},
"A5"!{+RD*+!LU{2}},
"A3";"A1"**@{-},
"A1";"A4"**@{-},
\endxy
,
\xy/r3pc/:
(0,0)*{\bullet}*\cir<5pt>{}!{+L*+!R{2}};(1.5,0)**@{-}*{\bullet}!{+R*+!L{3}}
\endxy
,
\xy/r3pc/:
{\xypolygon4"A"{\bullet}},
"A1"!{+RU*+!LD{5}},
"A2"!{+LU*+!RD{4}},
"A3"!{+LD*+!RU{3}*\cir<5pt>{}},
"A4"!{+RD*+!LU{6}},
"A1";"A3"**@{-},
"A2";"A4"**@{-},
\endxy
,
\xy/r3pc/:
{\xypolygon3"A"{\bullet}},
"A1"!{+U*+!D{7}},
"A2"!{+LD*+!RU{6}*\cir<5pt>{}},
"A3"!{+RD*+!LU{8}},
\endxy
\]
pointed at the vertex labelled $0$.
\end{Example}

We are now ready to prove the main theorem of this section:
\done{Martin: it is easy to refine the count further by the number of
  wings (including the degenerate wing).  However, I'm not sure it's
  worth the effort of explaining the details, so I commented the
  statement and the derivation out.}
\begin{Theorem}
  The number of torsion pairs in the cluster tube $\mathsf{C}_n$ is
  $$
  \T_n = \sum_{\ell\geq 0}%
  2^{\ell+1}\binom{n-1+\ell}{\ell}\binom{2n-1}{n-1-2\ell}.
  $$
  More precisely, the number of torsion pairs with a total of $k$
  triangles, $\ell$ cliques, and $m$ empty cells equals 
  $$
  \T_{n, k, \ell, m} = 2\binom{n-1+k+\ell+m}{n-1,k,\ell,m}\binom{n-1-k-\ell-m}{\ell+m}.
  $$
\end{Theorem}
\begin{proof}
  By Lemma~\ref{lem:bij-tors-pairs} and Remark~\ref{rmk:ordinary-gf}
  we want to compute the coefficient of $z^n$ in the exponential
  generating function of $(C\circ \P)^\bullet$
  which equals
  $z\left(\log\left(\frac{1}{1-\P(z)}\right)\right)^\prime$.  
  Since $z\left(\sum_{n\geq 0} a_n z^n\right)^\prime = \sum_{n\geq 0}
  n a_n z^n$, this is the same as $n$ times the coefficient of $z^n$
  in $\log\left(\frac{1}{1-\P(z)}\right)$, which we compute using
  Lagrange inversion.

  Recall from \cite[sec. 2, proof of thm. 1.3]{Kluge-Rubey} that 
  the generating function $\P(z)$
  satisfies the following algebraic equation:
  $$
  \P(z) = z + x\P(z)^2 + (y_1+y_2)\frac{\P(z)^3}{1-\P(z)},
  $$
  or equivalently
  $$
  \P(z)\left(1- x\P(z) - (y_1+y_2)\frac{\P(z)^2}{1-\P(z)}\right) = z.
  $$
  Thus $Q(z) = z\big(1- x z - (y_1+y_2)\frac{z^2}{1-z}\big)$ is the
  compositional inverse of $\P(z)$, i.e. $Q\left(\P(z)\right) = z$.

  Note that $\left(\log\frac{1}{1-z}\right)^\prime = \frac{1}{1-z}$.
  By writing $[z^n] f(z)$ for the coefficient of
  $z^n$ in $f(z)$, we therefore have
  \begin{equation*}
    n [z^n] \log\left(\frac{1}{1-\P(z)}\right) = %
    [z^{n-1}] \frac{1}{1-z} \frac{z^n}{Q(z)^n}.
  \end{equation*}
  Applying the multinomial theorem, 
  $$(1-X-Y_1-Y_2)^{-n}=%
  \sum_{k,\ell,m}\binom{n-1+k+\ell+m}{n-1,k,\ell,m} X^k Y_1^\ell
  Y_2^m
  $$
  we obtain by setting $X=xz$, $Y_1=y_1\frac{z^2}{1-z}$ and
  $Y_2=y_2\frac{z^2}{1-z}$:
  \begin{align*}
    \frac{1}{1-z} \frac{z^n}{Q(z)^n} &= %
    \sum_{k,\ell, m} \binom{n-1+k+\ell+m}{n-1,k,\ell,m} %
    (xz)^k y_1^\ell y_2^m \frac{z^{2(\ell+m)}}{(1-z)^{\ell+m+1}}\\
    &=
    \sum_{k,\ell, m, i} \binom{n-1+k+\ell+m}{n-1,k,\ell,m} %
    \binom{\ell+m+i}{l+m} x^k y_1^\ell y_2^m t^j
    z^{i+k+2(\ell+m)}.
  \end{align*}
  To extract the coefficient of $z^{n-1}$ we set $i =
  n-1-k-2(\ell+m)$.  
  Finally, we have to multiply the result by two, since either of the
  two halves of a torsion pair could be the one with finitely many
  indecomposable objects.

  To obtain the total number of torsion pairs we have to set
  $x=y_1=y_2=1$ in $\P(z)$ and $Q(z)$.  For the expansion of
  $\frac{1}{1-z}\frac{z^n}{Q(z)}$ it is then convenient to write the
  compositional inverse as $Q(z) =
  z(1-z)\big(1-\frac{2z^2}{(1-z)^2}\big)$.  We leave the remaining
  details to the reader, which are completely analogous to the 
  computation above. 
\end{proof}
\begin{Remark}
  \label{rmk:asymptotics}
  Since $\P(z)$ satisfies an algebraic equation, so does 
  \begin{equation*}
    (C\circ \P)^\bullet(z)=
    z\left(\log\left(\frac{1}{1-\P(z)}\right)\right)^\prime
    = z\frac{\P^\prime(z)}{1-\P(z)}.
  \end{equation*}
  Therefore, the asymptotic behaviour of its coefficients can be
  extracted automatically, for example using the {\tt equivalent}
  function in Bruno Salvy's package {\bf gdev} available at
  \url{http://algo.inria.fr/libraries/}.  Thus, we learn that the
  leading term of the asymptotic expansion of $\T_n = [z^n] 2(C\circ
  \P)^\bullet(z)$ is
  $$
  \frac{\alpha}{\sqrt{\pi n}} \rho^n,
  $$
  where $\rho=6.847333996370022\dots$ is the largest positive root of
  $8x^3 - 48x^2 - 47x + 4$ and $\alpha=0.2658656601482029\dots$ is
  the smallest positive root of $71x^6 + 213x^4 - 72x^2 + 4$.
\end{Remark}

Because the combinatorial construction of torsion pairs is so nice,
we can also compute the number of those torsion pairs
that are invariant under
$d$-fold application of Auslander-Reiten translation:
\begin{Proposition}
  \label{prop:invariant}
  For any $d\in \mathbb{N}$ dividing $n$, the number of torsion 
  pairs in the cluster tube
  $\mathsf{C}_n$ that are invariant under $d$-fold application of the
  Auslander-Reiten translation $\tau$ equals the number of torsion
  pairs in the cluster tube $\mathsf{C}_{n/d}$.
\end{Proposition}
\begin{proof}
  Applying $\tau$ to a torsion pair in the cluster tube
  $\mathsf{C}_{n}$ corresponds to shifting the associated
  $n$-periodic Ptolemy diagram on the $\infty$-gon by one vertex.
  Thus, a torsion pair invariant under $\tau^d$ corresponds to a
  $n/d$-periodic Ptolemy diagram on the $\infty$-gon, which in turn
  corresponds to a torsion pair in the cluster tube
  $\mathsf{C}_{n/d}$, by Theorem \ref{thm:torsion}. 
\end{proof}

\begin{Corollary}
  The number of torsion pairs in the cluster tube $\mathsf{C}_n$ up
  to Auslander-Reiten translation is
  $$
  \frac{1}{n}\sum_{d|n, d|k, d|\ell, d|m} \phi(d) \T_{n/d, k/d,
    \ell/d, m/d}.
  $$ 
\end{Corollary}
\begin{proof}
  For every $d|n$ there are $\phi(d)$ elements of order $d$ in the
  cyclic group generated by $\tau$, and for each group element there
  are $\T_{n/d,k/d,\ell/d,m/d}$ torsion pairs that are invariant.
  The corollary now follows from the Cauchy-Frobenius formula for the
  number of orbits,
  $$
  \frac{1}{n}\sum_{b=0}^{n} \#\{\text{torsion pairs fixed by
    $\tau^b$}\},
  $$
  and Proposition~\ref{prop:invariant}.
\end{proof}

Finally we remark that the enumerative results above can be phrased
as a \Dfn{cyclic sieving phenomenon}, which involves a finite set
$\Set X$, a cyclic group $C$ of order $n$ acting on $\Set X$, and a
polynomial $X(q)$ with non-negative integer coefficients:
\begin{Definition}
  The triple $\big(\Set X,C,X(q)\big)$ \Dfn{exhibits the cyclic
    sieving phenomenon} if for every $c\in C$ we have
  $$
  X(\omega_{o(c)})=\size{\Set X^c},
  $$
  where $o(c)$ denotes the order of $c\in C$, $\omega_d$ is a
  $d$\textsuperscript{th} primitive root of unity and $\Set
  X^c=\left\{x\in \Set X:c(x)=x\right\}$ denotes the set of fixed
  points of $\Set X$ under the action of $c\in C$.
\end{Definition}
In particular, $X(1)=\size{\Set X}$, i.e.\ $X(q)$ is a
$q$-analogue of the generating function for $\Set X$.  For the
statement of this particular instance of the cyclic sieving
phenomenon, we need to define $q$-binomial and $q$-multinomial
coefficients:
\begin{Definition}
  For $0\le k \le n$ the \Dfn{$q$-binomial coefficient} is
  $$
  \qbinom{n}{k}=\frac{\qi{n}!}{\qi{k}!\qi{n-k}!},
  $$
  where $\qi{n}!=\qi{n}\qi{n-1}\cdots\qi{1}$ and $\qi n =
  1+q\dots+q^{n-1}$.  Analogously, the \Dfn{$q$-multinomial
    coefficient} is
  $$
  \qbinom{n_1+n_2+\dots n_\ell}{n_1, n_2,\dots,
    n_\ell}=\frac{\qi{n_1+n_2+\dots n_\ell}!}{\qi{n_1}!\qi{n_2}!\dots
    \qi{n_\ell}!},
  $$
  where $n_1, n_2,\dots n_\ell$ are non-negative integers.
\end{Definition}

\begin{Theorem}\label{thm:sieving}
  Let $\T_{n,k,\ell,m}$ be the set of torsion pairs in the cluster
  tube $\mathsf{C}_n$ with $k$ triangles, $\ell$ cliques and $m$
  empty cells.  Let $\tau$ be the action of Auslander-Reiten
  translation, and let
  $$
  \T_{n,k,\ell,m}(q)=
  2\qbinom{n-1+k+\ell+m}{n-1,k,\ell,m}%
  \qbinom{n-1-k-\ell-m}{\ell+m}.
  $$

  Then $\big(\T_{n,k,\ell,m}, \langle\tau\rangle,
  \T_{n,k,\ell,m}(q)\big)$ exhibits the cyclic sieving phenomenon.
\end{Theorem}

The proof, to be given below, is an 
application of the $q$-Lucas theorem, see, for
example, \cite[thm.\ 2.2]{MR1176155}:
\begin{Lemma}[$q$-Lucas theorem]
  Let $\omega$ be a primitive $d$\textsuperscript{th} root of unity
  and $a$ and $b$ non-negative integers.  Then
  \begin{equation*}
    \qbinom[\omega]{a}{b}=%
    \binom{\lfloor\frac{a}{d}\rfloor}{\lfloor\frac{b}{d}\rfloor}
    \qbinom[\omega]{a-d\lfloor\frac{a}{d}\rfloor}{b-d\lfloor\frac{b}{d}\rfloor}.
  \end{equation*}
  In particular, if $b\equiv 0\pmod d$ then
  \begin{equation*}
    \qbinom[\omega]{a}{b}=%
   \binom{\lfloor\frac{a}{d}\rfloor}{\lfloor\frac{b}{d}\rfloor}.
  \end{equation*}
\end{Lemma}

\begin{proof}[Proof of Theorem~\ref{thm:sieving}]
  Suppose that $d|n$.  We have to evaluate $\T_{n,k,\ell,m}(q)$ at
  $q=\exp(2\pi i/d)$.  Suppose first that $d|k$, $d|\ell$ and $d|k$.
  Then, expressing the multinomial coefficient as a product of
  binomial coefficients and applying the $q$-Lucas theorem, we have
  \begin{align*}
    \T_{n,k,\ell,m}(q)%
    =&2\qbinom{n-1+k+\ell+m}{k}%
    \qbinom{n-1+\ell+m}{\ell}%
    \qbinom{n-1+m}{m}\\%
    &\quad\qbinom{n-1-k-\ell-m}{\ell+m}\\
    =&2\binom{\bar n+\bar k+\bar\ell+\bar m-1}{\bar k}%
    \binom{\bar n+\bar\ell+\bar m-1}{\bar\ell}%
    \binom{\bar n+\bar m - 1}{\bar m}\\%
    &\quad\binom{\bar n-\bar k-\bar\ell-\bar m - 1}{\bar\ell+\bar m}\\
    =& \binom{\bar n-1+\bar k+\bar \ell+\bar m}{\bar n-1, \bar k,
      \bar\ell, \bar m}%
    \binom{\bar n-\bar k-\bar\ell-\bar m - 1}{\bar\ell+\bar m},
  \end{align*}
  where we put $\bar n = n/d$, $\bar k = k/d$, $\bar \ell = \ell/d$
  and $\bar m = m/d$.

  Let us now consider the case that $d$ does not divide all of $k$,
  $\ell$ and $m$.  We show that in this situation
  $\qbinom{n-1+k+\ell+m}{n-1,k,\ell,m}$ vanishes.  Since this
  expression is symmetric in $k$, $\ell$ and $m$ it is sufficient to
  consider the case where $d$ does not divide $m$.  Furthermore,
  expanding the multinomial coefficient in binomial coefficients as
  above, it is sufficient to show that $\qbinom{n-1+m}{m}$ vanishes.
  Suppose that $m\equiv \alpha\pmod d$ with $0 < \alpha < d$.  Then
  $d\lfloor\frac{n-1+m}{d}\rfloor = n/d + \alpha$ and
  $d\lfloor\frac{m}{d}\rfloor = m-\alpha$.  Thus, by the $q$-Lucas
  theorem we obtain
  $$
  \qbinom{n-1+m}{m} =%
  \binom{\lfloor\frac{n-1+m}{d}\rfloor}{\lfloor\frac{m}{d}\rfloor}%
  \qbinom{\alpha-1}{\alpha}
  $$
  which is indeed zero.
\end{proof}

\end{document}